\documentclass[12pt]{amsart}

\usepackage{amssymb,amsmath,latexsym}
\usepackage{amsthm}
\usepackage{amsfonts}
\usepackage{enumerate}
\usepackage{booktabs}

\usepackage{hyperref}

\usepackage[absolute]{textpos}

\usepackage{fancyhdr}

\newtheorem{thm}{Theorem}[section]
\newtheorem{defn}{Definition}[section]
\newtheorem{prop}{Proposition}[section]
\newtheorem{lemma}{Lemma}[section]

\newcommand{\Rmnum}[1]{ \uppercase\expandafter{\romannumeral  #1}}

\numberwithin{equation}{section}

\newcommand{\mysectionname}{}
\newcommand{\newsection}[1]{\section{#1}\renewcommand{\mysectionname}{\uppercase{#1}}}

\renewcommand{\headrulewidth}{0pt}

\begin{document}
\title{On series of free $R$-diagonal operators}
\author{Hari Bercovici and Ping Zhong}
\address{Department of Mathematics, Rawles Hall, 831 East Third Street, Indiana University, Bloomington, Indiana 47405, U.S.A. }
\email{bercovic@indiana.edu}
\email{pzhong@indiana.edu}
\begin{abstract}
For a series of free $R$-diagonal operators, we prove an analogue of
the three series theorem. We show that a series of free $R$-diagonal
operators converges almost uniformly if and if two numerical series
converge.
\end{abstract}

\maketitle
\newsection{Introduction}
Free probability theory, introduced by Voiculescu, has many
similarities with the classic probability theory. In this theory,
free independence replaces the usual one. The free central limit
theorem was proved by Voiculescu in \cite{[DVV1985]}. Later, other
limit laws were studied by different authors. 
Pata and the first-named author proved that there is one-to-one correspondence of the
weak limits between the free case and the classic case
in \cite{[BP1999],[BP2000]}. 
One might be interested in
other types of limit theorems in free probability. In
\cite{[B2005]}, the first-named author proved an analogue of the three series
theorem, which characterizes the almost uniform convergence of a
series of random variables by the convergence of three numerical
series. In this article, we consider series of $R$-diagonal
operators. We prove that the uniform convergence of a series of
$R$-diagonal operators is equivalent to the convergence of two
numerical series. To prove this result, we use the tools introduced
recently by Nica and Noyes and the first-named author.

We review some results of \text{\footnotesize{$\boxplus$}}-convolution in the Section
\ref{boxplus}. We recall the definition of the
\text{\footnotesize{$\boxplus_{RD}$}}-convolution, the properties of this convolution and
prove some relevant theorems in the Section \ref{RDconvolution}. We
prove our main theorem in the Section \ref{convergence}.
\section{Free independence and the \text{\footnotesize{$\boxplus$}}-convolution}\label{boxplus}
We use the notation $(\mathcal{M}, \tau)$ to denote a tracial
non-commutative W*-probability space, where $\mathcal{M}$ is a von
Neumann algebra, and $\tau$ is a faithful normal tracial state. 
We assume that $\mathcal{M}$ acts on a Hilbert space $H$. 
The collection of closed, densely defined operators affiliated with
$\mathcal{M}$ is a *-algebra, which we denote by
$\widetilde{\mathcal{M}}$. An operator in $\widetilde{\mathcal{M}}$
is also called a random variable. If $X$ is a closed densely defined
operator on $H$ and $X$ has the polar decomposition $X=u|X|$, then
$X$ is affiliated with $\mathcal{M}$ if and only if $u$ and the
spectral projections $e_{\lambda}$ of $|X|$ are in $\mathcal{M}$. If
$X$ is a self-adjoint operator, the distribution $\mu_X$ of $X$ is
the probability measure on $\mathbb{R}$ so that
$\mu_X(\sigma)=\tau(e_X(\sigma))$.

In Voiculescu's free probability theory, a family subalgebras
$(\mathcal{M}_i)_{i\in I}$ of $\mathcal{M}$ is said to be free if,
given a natural number $n$ and elements $X_1 \in \mathcal{M}_{i_1}$,
$X_2 \in \mathcal{M}_{i_2}$, $\cdots$, $X_n \in \mathcal{M}_{i_n}$,
such that $\tau(X_1)=\tau(X_2)=\cdots=\tau(X_n)=0$ and $i_j\neq
i_{j+1}$ for $1\leq j <n$, we have $\tau(X_1 X_2\cdots X_n)=0$.
Consider operators $X_i=U_i|X_i| \in \widetilde{\mathcal{M}}$, $i\in I$,
the family $(X_i)_{i\in I}$ is said to be *-free if the family $(A_i)_{i\in I}$
is a free family of algebra, where $A_i$ is the smallest unital von Neumann subalgebra
containing $U_i$ and the spectral projections of $|X_i|$.
The free additive convolution \text{\footnotesize{$\boxplus$}} is an operation on probability
measures on the real line such that, if $X_1$ and $X_2$ are free
self-adjoint random variables, we have
$\mu_{X_1+X_2}=\mu_{X_1}$\text{\footnotesize{$\boxplus$}}$ \mu_{X_2}$.

Now we review the calculation of the free additive convolution. Denote by $\mathbb{C}$ the complex plane,
and $\mathbb{C}^+=\{z\in \mathbb{C}: \Im{z}>0 \}$. Given a probability measure
$\mu$ on $\mathbb{R}$, the Cauchy transform of $\mu$ is defined as
\begin{equation}\nonumber
G_\mu(z)=\int_{-\infty}^{\infty}\frac{d\mu(t)}{z-t}, \hspace{6pt} z \in \mathbb{C}^+ .
\end{equation}
Given two positive numbers $\alpha$ and $\beta$, let us denote
$\Gamma_{\alpha,\beta}=\{z=x+iy \in \mathbb{C}^+: y>\alpha|x|, |z|
>\beta \}$. Set $F_{\mu}(z)=1/G_\mu(z)$, then
$F^{-1}_\mu(z)$ exists in some domain of the form
$\Gamma_{\alpha,\beta}$, where $F^{-1}_\mu(z)$ is the right inverse of
$F_\mu(z)$ with respect to composition. 
The Voiculescu transform of
$\mu$ is defined as $\varphi_\mu(z):=F^{-1}_\mu(z)-z$. Let $\mu_1,
\mu_2$ be two probability measures on $\mathbb{R}$, and $\mu_1$\text{\footnotesize{$\boxplus$}}$\mu_2$
be their free additive convolution. A remarkable
property of the Voiculescu transform is
\begin{equation}\nonumber
\varphi_{\mu_1 \boxplus
\mu_2}(z)=\varphi_{\mu_1}(z)+\varphi_{\mu_2}(z).
\end{equation}
The above equation is valid in some domain of the form
$\Gamma_{\alpha,\beta}$. The $\mathcal{R}$-transform of $\mu$ is an
analytic function defined as
$\mathcal{R}_{\mu}(z)=\varphi_\mu(1/z)$.

Weak convergence of probability measure can be translated in terms of the corresponding $\varphi$-functions.
The following theorem is from \cite{[BV1993]}.
\begin{thm}\label{conv1}
Let $\{\mu_n\}_{n=1}^{\infty}$ be a sequence of probability measures on $\mathbb{R}$. The following assertions are equivalent.
\begin{enumerate}[$(1)$]
\item\label{1} $\mu_n \rightarrow \mu$ weakly,
\item\label{2} There exist $\alpha, \beta>0$ and a function $\varphi$ such that $\varphi_{\mu_n}\rightarrow \varphi$
uniformly on the compact subsets of $\Gamma_{\alpha,\beta}$.
\end{enumerate}
Moreover, if $(1)$ and $(2)$ are satisfied, we have
$\varphi=\varphi_\mu$ in $\Gamma_{\alpha,\beta}$.
\end{thm}

Almost uniform convergence is an appropriate
replacement of almost sure convergence in classical probability
theory. We say that a sequence $X_n$ in $\widetilde{\mathcal{M}}$
converges almost uniformly to zero, if for every $\varepsilon >0$,
there is a projection $p\in \mathcal{M}$ with $\tau(1-p) <
\varepsilon$ such that $\|X_n p \| \rightarrow 0$ as $n \rightarrow
\infty$.
The measure topology extends the classical concept of convergence in probability.
This topology was introduced by Nelson in \cite{[Nel1974]}, one may also
find an excellent exposition in \cite{[Takesaki03]}.
\begin{defn}
The measure topology of $\mathcal{M}\ ($with respect to $\tau)$ is the uniform topology given a neighborhood system
$\{x+N(\varepsilon, \delta) : \varepsilon, \delta >0 \}$, $x \in \mathcal{M}$ where
$N(\varepsilon, \delta)=\{a \in \mathcal{M} : \|ap\|<\varepsilon$ and  $\tau(p^{\perp}) \leq \delta$
for some projection p in $\mathcal{M}\}$. We say $X_n$ converges in measure to $X$ if $X_n$ converges to $X$
in this topology.
\end{defn}

It is clear that convergence in measure is weaker than almost uniform convergence. When $X_n$, $X$ are self-adjoint
random variables, then $X_n \rightarrow X$ in measure implies $\mu_{X_n} \rightarrow \mu_X$ weakly.

\section{$R$-diagonal operators and the \text{\footnotesize{$\boxplus_{RD}$}}-convolution}\label{RDconvolution}
The class of $R$-diagonal elements in free probability was introduced by Nica and Speicher 
in \cite{[NS1997]}.  
This notion has a natural extension
to the unbounded case. Let us recall the following definitions in \cite{[HS2007]}.

\begin{defn}
\begin{enumerate}[$(1)$]
\item Let $S, T \in \widetilde{\mathcal{M}}$, we say $S$ and $T$ have the same *-distribution if there exists a trace-preserving
*-isomorphism $\phi$ from $W^*(S)$ onto $W^*(T)$ with $\widetilde{\phi}(S)=T$, where $\widetilde{\phi}$ is the natural extension of
$\phi$ to unbounded operators.
\item $T\in \widetilde{\mathcal{M}}$ is said to be $R$-diagonal if there exist a von Neumann algebra $\mathcal{N}$, with
a faithful, normal, tracial state and *-free elements $U$ and $H$ in $\widetilde{\mathcal{N}}$, such that $U$ is
Haar unitary, $H \geq 0$, and such that $T$ has the same *-distribution as $UH$.
\end{enumerate}
\end{defn}

It is well-known that if $T\in \widetilde{\mathcal{M}}$ is
$R$-diagonal with ker$(T)=0$, and the partial decomposition of $T$ is
$T=u|T|$, then $u$ is a Haar unitary which is *-free from $|T|$. The
distribution of an $R$-diagonal operator $T$ is uniquely determined by $|T|$. Given two
$R$-diagonal operators $X, Y \in \widetilde{\mathcal{M}}$, such that
$X$ and $Y$ are *-free, then the operator $X+Y \in
\widetilde{\mathcal{M}}$ is also $R$-diagonal. In \cite{[BN2012]},
Nica and Noyes and the first-named author introduced a new kind of convolution of
probability measures on the positive real line $\mathbb{R}^+$, which
they denoted by \text{\footnotesize{$\boxplus_{RD}$}}. The defining property of this new
convolution is $\mu_{Z^*Z}=\mu_{X^*X}$\text{\footnotesize{$\boxplus_{RD}$}}$\mu_{Y^*Y}$,
if $X$ and $Y$ are *-free $R$-diagonal operators and $Z=X+Y$.
We will follow the paper \cite{[BN2012]} to describe the functional
which linearizes \text{\footnotesize{$\boxplus_{RD}$}}-convolution.

For convenience, given two positive numbers $\alpha$ and $\beta$, we
denote $\Lambda_\alpha=\{z=x+iy \in \mathbb{C}: x>0,
|y|<\alpha|x|\}$ and $\Lambda_{\alpha, \beta}=\{z \in
\mathbb{C}:|z|>\beta \} \backslash \Lambda_\alpha$. Let $\mu$ be a
probability measure $\mu$ on $[0, +\infty)$, then its Cauchy transform $G_\mu$
is defined in $\mathbb{C}\backslash [0, +\infty)$ and its
Voiculescu transform $\varphi_\mu$ is defined in some domain of the form
$\Lambda_{\alpha, \beta}$. It is known that for $z$ in a domain $\Lambda_{\alpha, \beta}$,
we have $\lim_{|z|\rightarrow \infty}\varphi_\mu(z)/z =0$.
We set
\begin{equation}\nonumber
V_\mu(z)=\frac{1}{z}\left(1+\frac{1}{z}\varphi_\mu \left(z \right)
\right),
\end{equation}
and let $W_\mu(z)=1/V_\mu(z)$, then $W_\mu^{-1}(z)$ exists in some domain of the
form $\Lambda_{\alpha, \beta}$. The $\varphi\varphi$-transform of
$\mu$ is defined by $\varphi\varphi_\mu(z)=W_\mu^{-1}(z)-z$. This transform
linearizes the \text{\footnotesize{$\boxplus_{RD}$}}-convolution, namely if $\mu_1$ and
$\mu_2$ are two probability measures on $[0, +\infty)$, then
$\varphi\varphi_{\mu_1\boxplus_{RD}\mu_2}(z)=\varphi\varphi_{\mu_1}(z)+\varphi\varphi_{\mu_2}(z)$
in some domain of the form $\Lambda_{\alpha, \beta}$. One might find
detailed description of this definition in \cite{[BN2012]}.

The weak convergence of probability measures is equivalent to convergence properties
of the corresponding $\varphi\varphi$-transforms.
\begin{thm}\label{conv2}
Let $\{\mu_n\}_{n=1}^{\infty}$ be a sequence of probability measures on $\mathbb{R}^+$. The following assertions are equivalent.
\begin{enumerate}[$(1)$]
\item\label{3} $\mu_n \rightarrow \mu$ weakly,
\item\label{4} There exist two positive numbers $\alpha$, $\beta$, a domain of the form $\Lambda_{\alpha, \beta}$
and a function $\varphi\varphi$ such that
$\varphi\varphi_{\mu_n}\rightarrow \varphi\varphi$
uniformly on the compact subsets of $\Lambda_{\alpha, \beta}$.
\end{enumerate}
Moreover, if $(1)$ and $(2)$ are satisfied, we have
$\varphi\varphi=\varphi\varphi_\mu$ in $\Lambda_{\alpha, \beta}$ and $sup_n|\varphi\varphi_{\mu_n}(z)|=o(z)$ as $z \rightarrow
\infty$ for $z \in$ $\Lambda_{\alpha, \beta}$.
\end{thm}

We would also like to recall some basic combinatorial theory of bounded $R$-diagonal operators.
If $X$ is a bounded $R$-diagonal operator, $X=a+bi$, where $a, b$ are self-adjoint operators, then $X$ is a $R$-diagonal operator
if and only if the joint $R$-transform of $(X, X^*)$ is a series $R_{X, X^*}$ in two non-commuting indeterminates $z$ and $z^*$
of the form
\begin{equation}\label{RRequ}
R_{X, X^*}(z, z^*)=\sum_{n=1}^{\infty}\alpha_n(zz^*)^n + \sum_{n=1}^{\infty}\alpha_n(z^*z)^n.
\end{equation}
In other words, the only non-vanishing cumulants of $(X, X^*)$ are
\begin{equation}\nonumber
k_{2n}(X,X^*,\cdots,X,X^*)=k_{2n}(X^*, X,\cdots, X^*, X)=\alpha_n, \,\,n \in \mathbb{N}.
\end{equation}
We denote $f_X(z):=\sum_{n=1}^{\infty}\alpha_n z^n$. As indicated in \cite{[BN2012]},
we have
\begin{equation}\label{fANDvarphi}
f_X(z)=\frac{1}{z}\varphi\varphi_{\mu_{X^*X}}\left(\frac{1}{z}\right).
\end{equation}

Given a bounded self-adjoint operator $a \in \mathcal{M}$, we set $R_a(z):=\sum_{n=1}^{\infty}k_n(a,\cdots,a)z^n$,
where $k_n$ is the $n$-th free cumulant of $a$.
When $a$ is the real part of $X$, then the odd moments of $\mu_a$ vanish, therefore $\mu_a$
is a symmetric measure on $\mathbb{R}$.
By the linearity property of the free cumulant functions and the property of $R$-diagonal
operators mentioned above, we have
\begin{align}\nonumber
k_{2n}(a,\cdots,a)
         &=k_{2n}\left(\frac{X+X^*}{2},\cdots,\frac{X+X^*}{2}\right)\\  \nonumber
                &=\frac{1}{2^{2n}}\left[k_{2n}(X,X^*,\cdots,X,X^*)+k_{2n}(X^*,X,\cdots,X^*,X)\right] \nonumber
\end{align}
and $k_{2n+1}(a,\cdots,a)=0$.

Therefore, we have
\begin{equation}\label{RandRR}
\begin{split}
R_a(z)&=\sum_{n=1}^{\infty}k_n(a,\cdots,a)z^n\\
        &=\sum_{n=1}^{\infty}2\alpha_n\left(\frac{z}{2}\right)^{2n}
        =2f_X\left( \left( \frac{z}{2} \right)^2 \right)
\end{split}
\end{equation}

We recall that the above function is a little different from the $\mathcal{R}$-transform
defined in the Section \ref{boxplus}, see \cite{[NS1996], [VDN1992]} for details. 
In fact, we have $R_a(z)=z\mathcal{R}_{\mu_a}(z)$.
From the equations (\ref{fANDvarphi}) and (\ref{RandRR}), we obtain that
\begin{equation}\label{phiAndphiphi}
\begin{split}
\varphi_{\mu_a}(z)&=\mathcal{R}_{\mu_a}\left(\frac{1}{z}\right)=zR_a\left(\frac{1}{z}\right) \\
             &=2z f_X\left(\left(\frac{1}{2z}\right)^2\right) \\
             &=(2z)^3\varphi\varphi_{\mu_{X^*X}}\left(\left( 2z \right)^2\right)
\end{split}
\end{equation}
The above properties of the real part of $R$-diagonal operators is also true for unbounded operators.
\begin{prop}\label{phiandphi}
Let $X$ be a $R$-diagonal operator, $a$ be its real part, then $\mu_a$ is a symmetric measure on $\mathbb{R}$.
Moreover, we have $\varphi_{\mu_a}(z)=(2z)^3\varphi\varphi_{\mu_{X^*X}}((2z)^2)$ in some domain of
the form $\Gamma_{\alpha, \beta}$.
\end{prop}
\begin{proof}
If $X$ is bounded, the proposition follows from our previous
discussion. For the unbounded case, we use the truncated operators
$X_n:=Xe_{|X|}([0,n))$ to approximate $X$. Let $a_n$ be the real
part of $X_n$, then the argument in \cite{[BV1993]} tells us that
the sequence $\mu_{a_n}$ converges to $\mu_a$ weakly and the
sequence $\mu_{X_n^*X_n}$ converges to $\mu_{X^*X}$ weakly. The the
result follows from the results of Theorems
\ref{conv1} and \ref{conv2}.
\end{proof}

Roughly, the function $\varphi\varphi_{\mu}$ is very close to the function $\varphi_{\mu}$ for a probability measure $\mu$ on the positive real line.
The following results are from \cite{[BN2012]}.
\begin{lemma}\label{tight}
Let $\{\mu_n\}_{n=1}^{\infty}$ be a sequence of probability measures on $\mathbb{R}^+$.
If $\varphi\varphi_{\mu_n}(z)=o(z)$ as $z \rightarrow \infty$ in some domain $\Lambda_{\alpha, \beta}$, then $\{ \mu_n \}_{n=1}^{\infty}$ is a tight family.
\end{lemma}

\begin{thm}\label{estimate}
For any probability measure $\mu$ on $\mathbb{R}^+$ we have
\begin{equation}\nonumber
\varphi\varphi_\mu(z)=z^2\left[G_\mu(z)-1/z \right](1+o(z))
\end{equation}
as $z
\rightarrow \infty$ within some domain $\Lambda_{\alpha, \beta}$.
Moreover, this estimate is uniform if $\mu$ varies in a tight family
of measures.
\end{thm}

\section{Convergence of series of $R$-diagonal operators}\label{convergence}
Now we prove our main theorem which is an analogue of the three series theorem Kolmogorov and L\'{e}vy. Note
that the condition $(4)$ in Theorem \ref{thm4.1} below involves only two numerical series, this is due to the circular symmetry of the $R$-diagonal operators.
\begin{thm}\label{thm4.1}
Let $(X_n)_{n=1}^{\infty}$ be a free sequence of $R$-diagonal operators.
The following conditions are equivalent.
\begin{enumerate}[$(1)$]
\item\label{01} $\sum_{n=1}^{\infty}X_n$ converges almost uniformly,
\item\label{02} $\sum_{n=1}^{\infty}X_n$ converges in measure,
\item\label{03} Let $S_n:=\sum_{k=1}^n X_k$, then the sequence of probability measures
$\{\mu_{S_n^*S_n}\}_{n=1}^{\infty}$ converges weakly,
\item\label{04} the following two series converge:
\begin{equation}\nonumber
\sum_{n=1}^{\infty}\tau(|X_n|^2 : |X_n|\leq 1), \sum_{n=1}^{\infty}\tau(e_{|X_n|}((1,+\infty))).
\end{equation}
\end{enumerate}
\end{thm}
\begin{proof}
Suppose (\ref{04}) is true, for $R$-diagonal operators $X_n$, we
have $\tau(X_n:|X_n|\leq1)=0$. (\ref{04}) $\Rightarrow$
(\ref{01}) follows from a result of Batty \cite{[Batty79], [Jajte1985]}.

(\ref{01})$\Rightarrow$ (\ref{02}) follows directly from
definitions. Assume we have (\ref{02}), by the properties of
convergence in measure given in \cite[Theorem 1]{[Nel1974]}, 
the adjoint operation and the joint multiplication are uniformly
continuous on bounded sets in the measure topology, thus (\ref{02})
implies $S_n^*S_n$ converges in measure. Therefore, (\ref{03}) is
true.

Now we concentrate on the proof of (\ref{03})$\Rightarrow$ (\ref{04}).
We have $S_n=\sum_{k=1}^n X_k$, and that the sequence of probability measures
$\{\mu_{S_n^*S_n}\}_{n=1}^{\infty}$ converges weakly.
Let us denote by $\mu_n$ the distribution of $X_n^*X_n$, and denote by $\nu_n$ the distribution of $S_n^*S_n$. We also
let $\nu$ be the limit distribution of $\nu_n$. Then by the result in Section 3 we have
\begin{equation}\label{sum}
\varphi\varphi_{\nu_n}=\varphi\varphi_{\mu_1}+\cdots+\varphi\varphi_{\mu_n}
\end{equation}
in some domain. By Theorem \ref{conv2}, there is a domain of the
form $\Lambda_{\alpha, \beta}$ where all the functions
$\varphi\varphi_{\nu_n}, \varphi\varphi_\nu$ are defined and
moreover $\lim_{n\rightarrow
\infty}\varphi\varphi_{\nu_n}(z)=\varphi\varphi_\nu(z)$ 
uniformly on the compact subsets of $\Lambda_{\alpha,\beta}$
and $\varphi\varphi_{\nu_n}(z)=o(z)$ uniformly as $z \rightarrow \infty$
in the domain of the form $\Lambda_{\alpha, \beta}$.

Let $a_n$ be the real part of $X_n$ and $b_n$ be the imaginary part of $X_n$,
and set $A_n=\sum_{k=1}^{n}a_k$.
The convergence of the functions $\varphi\varphi_{\nu_n}$ to the function $\varphi\varphi_{\nu}$
and Proposition \ref{phiandphi} imply that the functions $\varphi_{\small{A_n}}$ converges to the 
function $\varphi_{\small{\Re (S)}}$ 
uniformly on compact subsets of a domain of the form $\Gamma_{\alpha',\beta'}$
for some operator $S$.
The main result of \cite{[B2005]} implies that $A_n$ converges almost uniformly.  
This implies that $a_n$ converges to $0$ in the measure topology and 
likewise $b_n$ converges to $0$ in the measure topology.
Therefore, $X_n^*X_n$ converges to $0$ in the measure topology and 
the probability measures $\{\mu_{n}\}_{n=1}^{\infty}$ form a tight family. 
By the estimate in Theorem \ref{estimate}, we have
\begin{equation}\label{ineqphi}
\left|z^2\left[G_{\mu_n}(z)-\frac{1}{z}\right]\right| \leq
2|\varphi\varphi_{\mu_n}(z)|
\end{equation}
for $z $ in the domain $\Lambda_{\alpha, \beta}$ which is large
enough.

For $z=-y$ with $y>0$, we have
\begin{align}\nonumber
z^2\left[G_{\mu_n}(z)-\frac{1}{z}\right]&=y^2\int_0^{\infty}\left[\frac{1}{-y-t}-\frac{1}{-y}\right]d\,\mu_n(t)\\
\nonumber
                 &=\int_0^{\infty}\frac{yt}{y+t}d\,\mu_n(t) \nonumber
\end{align}
which is always positive. 
On the other hand, by Proposition \ref{phiandphi} we have
$\varphi_{\mu_{a_n}}(z)=(2z)^3\varphi\varphi_{\mu_n}((2z)^2)$
where $\mu_{a_n}$ is the distribution of $a_n$. By the same
proposition, $\mu_{a_n}$ is symmetric, we have
$G_{\mu_{a_n}}(z)=-\overline{G_{\mu_{a_n}}(-\overline{z})}$,
therefore we obtain
\begin{equation}\nonumber
\varphi_{\mu_{a_n}}(z)=-\overline{G_{\mu_{a_n}}(-\overline{z})}
\end{equation}
in the domain where $\varphi_{\mu_{a_n}}$ is defined. In particular,
$\varphi_{\mu_{a_n}}(i\sqrt{y}/2)$ is purely imaginary for $y>0$.
This implies $\varphi\varphi_{\mu_n}(-y)$ is always negative.

We choose $y$ large enough such that when $z=-y$ we have estimate as
in (\ref{ineqphi}). Moreover, the terms in the right hand side of
the equation (\ref{sum}) are all negative, and their summation is
convergent. Therefore, the series
\begin{equation}\nonumber
\sum_1^{\infty}\int_0^{\infty}\frac{yt}{y+t}d\,\mu_n(t)
\end{equation}
is also convergent. Now we obtain the inequality,
\begin{equation}\nonumber
\sum_{n=1}^{\infty}\tau(|X_n|^2 : |X_n|\leq 1) +
\sum_{n=1}^{\infty}\tau(e_{|X_n|}((1,+\infty))) \leq
(y+1)\int_0^{\infty}\frac{t}{y+t}d\mu_n(t),
\end{equation}
which implies (\ref{04}). The proof is complete.
\end{proof}
\section*{Acknowledgments}
The first author was partially supported by a grant from the National Science Foundation.

\end{document}